\documentclass{amsart}
\usepackage{amsmath,amssymb,amsthm,url,comment}
\usepackage{graphicx,color,array}
\usepackage[all]{xy}

\newtheorem{theorem}{Theorem}[section]  
\newtheorem{lemma}[theorem]{Lemma} 
 
\newtheorem{proposition}[theorem]{Proposition} 
 
\newtheorem{question}{Question} 
\newtheorem{corollary}[theorem]{Corollary}
\theoremstyle{definition}
\newtheorem{definition}[theorem]{Definition}

\newtheorem*{remark*}{Remark}

\newcommand{\Z}{\mathbb{Z}}
\newcommand{\R}{\mathbb{R}}
\newcommand{\C}{\mathbb{C}}

 \makeatletter
    
   \@addtoreset{equation}{section}
  \makeatother

\begin{document}

\title{Successively almost positive links }

\newcommand{\Crossing}{
\raisebox{-3mm}{
\begin{picture}(24,28)
\put(0,0){\line(1,1){10}}
\put(0,24){\vector(1,-1){24}}
\put(14,14){\vector(1,1){10}}
\end{picture} } 
}

\newcommand{\Smooth}{
\raisebox{-3mm}{
\begin{picture}(24,28)
\qbezier(0,0)(12,14)(24,0)
\qbezier(0,24)(12,10)(24,24)
\end{picture}} 
}

\newcommand{\LCross}{
\raisebox{-3mm}{
\begin{picture}(24,28)
\put(0,0){\line(1,1){24}}
\put(0,24){\line(1,-1){24}}
\end{picture}}
}

\author{Tetsuya Ito}
\address{Department of Mathematics, Kyoto University, Kyoto 606-8502, JAPAN}
\email{tetitoh@math.kyoto-u.ac.jp}

\subjclass[2020]{Primary~57K10}

\begin{abstract}
As an extension of positive or almost positive diagrams and links, we introduce a notion of successively almost positive diagrams and links, and good successively almost positive diagrams and links. We review various properties of positive links or almost positive links, and explain how they can be extended to (good) successively almost positive links. Our investigation also leads to an improvement of known results of positive or almost positive links.
\end{abstract}
\maketitle

\section{Introduction}

A link diagram is \emph{positive} if all the crossings are positive, and a \emph{positive link} is a link that can be represented by a positive diagram. Positive links have various nice properties and form an important class of links.

An innocent generalization of a positive diagram is a \emph{$k$-almost positive diagram}, a diagram such that all but $k$ crossings are positive. A $1$-almost positive diagram is usually called an \emph{almost positive diagram} and has been studied in various places.

It is known that almost positive links share various properties with positive links. As is discussed in \cite{pt} there are special properties of $2$ or $3$-almost positive links. However when $k$ is large, $k$-almost positive links fail to have nice properties similar to positive links because every knot $K$ is $k$-almost positive for sufficiently large $k$.

The aim of this paper is to propose a better and more natural generalization of a (almost) positive diagram which we call a \emph{successively ($k$-)almost positive diagram}. We also introduce a \emph{good successively ($k$-)almost positive diagram} which is a successively positive diagram having an additional condition. We show (good) successively almost positive links share various properties with (almost) positive links, for \emph{all} $k\geq 0$.

The paper has an aspect of survey of (almost) positive links; we review various properties of (almost) positive links which can be found in various places, sometimes in a different context or prospect\footnote{For example, properties of positive links discussed in \cite{cr} was obtained as a corollary of properties of more general class of links called \emph{homogeneous links}.}. 
We try to make it clear how the proof of the properties of (almost) positive links utilizes or reflects the positivity of diagrams, and explain how to extend the proof to a successively almost positive case.

Besides the extension of known properties, our investigation sometimes leads to an improvement of known results or new results of (almost) positive links themselves. See Corollary \ref{cor:HOMFLY-almost-positive}, Corollary \ref{cor:sig-vs-euler}, and Corollary \ref{cor:finite-almost-positive} and Theorem \ref{theorem:almost-type-II}.

In the following, for a link diagram $D$ we use the following notations.

\begin{itemize}
\item $s(D)$ is the number of Seifert circles.
\item $c(D)$ is the number of crossings, and $c_{\pm}(D)$ is the number of positive and negative crossings.
\item $w(D)(=c_+(D)-c_-(D))$ is the writhe.
\item $S_D$ is the \emph{canonical Seifert surface} of $D$, the Seifert surface obtained from $D$ by applying Seifert's algorithm.
\item $\chi(D)=s(D)-c(D)=\chi(S_D)$ is the \emph{euler characteristic} of $D$. When $K$ is a knot, we often use the \emph{genus} $g(D)=\frac{1-\chi(D)}{2}$ instead of $\chi(D)$.
\end{itemize}

\subsection{Successively $k$-almost positive diagram}

\begin{definition}[Successively $k$-almost positive diagram/link]
A knot diagram $D$ is \emph{successively $k$-almost positive} if all but $k$ crossings of $D$ are positive, and the $k$ negative crossings appear successively along a single overarc (see Figure \ref{fig:succ-almost-positive}).
\end{definition}

\begin{figure}[htbp]
\includegraphics*[width=45mm]{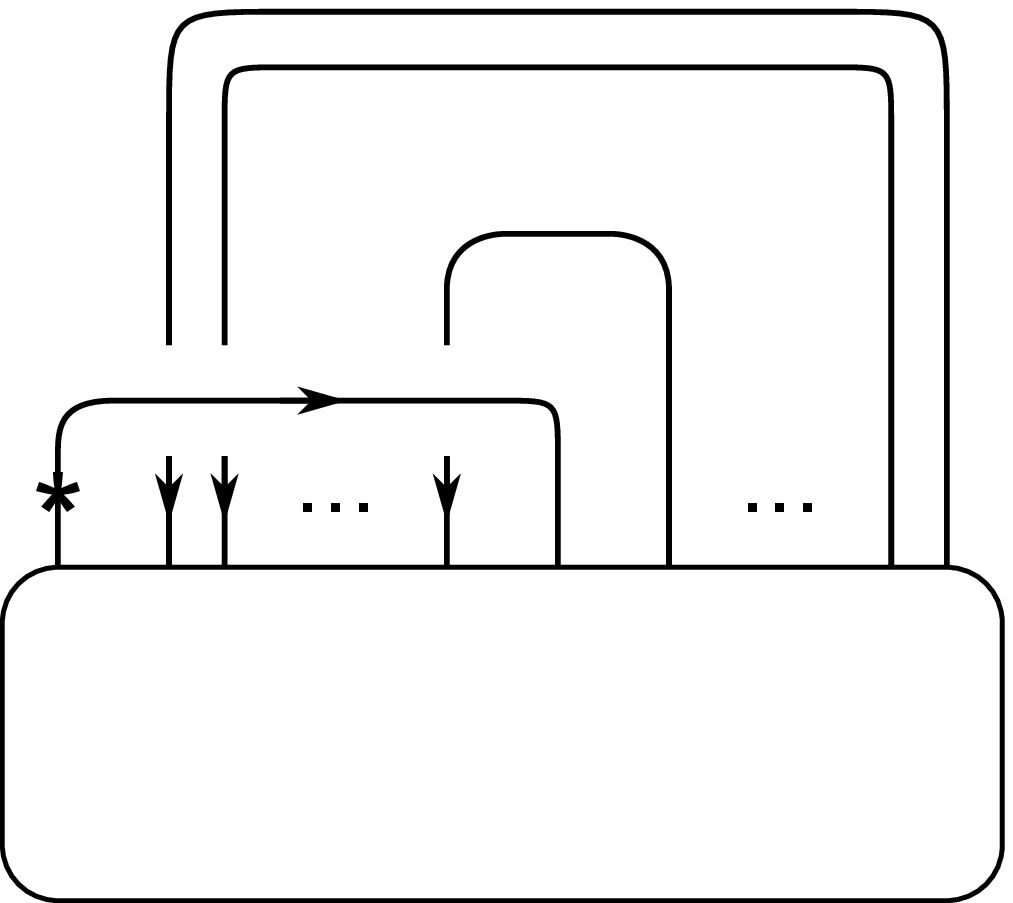}
\begin{picture}(0,0)
\put(-110,20) {\large Positive diagram}
\end{picture}
\caption{Successively $k$-almost positive diagram. $\ast$ represents the standard base point, which will be used in Section \ref{section:skein}.} 
\label{fig:succ-almost-positive}
\end{figure} 

This type of diagram appeared in \cite[Theorem 5.3]{imt} as an extension of positive diagram, designed so that the technique of constructing generalized torsion elements can be applied.

The definition says that a successively $0$-almost positive (resp. successively $1$-almost positive) diagram is nothing but a positive (resp. almost positive) diagram. On the other hand, a $2$-almost positive diagram is not necessarily a successively $2$-almost positive diagram; the two-crossing diagram of the negative Hopf link is $2$-almost positive but not successively $2$-almost positive.

As we will see and discuss later, an investigation of good properties of diagrams themselves leads to the following restricted class of successively $k$-almost positive diagrams.

\begin{definition}[Good successively $k$-almost positive diagram/link]
\label{definition:good}
A $k$-successively almost positive diagram $D$ is \emph{good} if, when two distinct Seifert circles $s,s'$ of $D$ are connected by a negative crossing then there are no other crossings connecting $s$ and $s'$ (see the left figure of Figure \ref{fig:MP-move2} for a schematic illustration).
\end{definition}

A \emph{successively almost positive link} is a link represented by a successively $k$-almost positive diagram for some $k$. Similarly, a \emph{good successively almost positive link} is a link represented by a good successively $k$-almost positive diagram for some $k$.

\subsection{Properties}

We show the following properties, generalizing (almost) positive link cases (\cite{cr},\cite{pt}).

\begin{theorem}
\label{theorem:signature}
If $K$ is successively $k$-almost positive, then its Levine-Tristram signature\footnote{We remark that we adopt the convention that the positive trefoil has negative signature, as opposed to the one adopted in \cite{bdl}.} $\sigma_{\omega}$ satisfies $\sigma_\omega(K)\leq 0$ for all $\omega \in S^{1} =\{z \in \C \: | \: |z|=1\}$.
\end{theorem}

\begin{theorem}
\label{theorem:conway-non-negative}
If $K$ is successively almost positive, the Conway polynomial $\nabla_K(z)$ is non-negative.
\end{theorem} 
Here we say that a polynomial $f(z)=a_0+a_1z+\cdots$ is \emph{non-negative} if $a_i\geq 0$ for all $i$. 

A positive diagram $D$ has the following nice properties\footnote{The first statement can be confirmed by looking at the linking number, and the second statement follows from the Bennequin inequality or \cite[Corollary 4.1]{cr}};
\begin{itemize}
\item $D$ represents a split link if and only if the diagram $D$ is split.
\item The canonical Seifert surface $S_D$ attains the maximum euler characteristic among its Seifert surfaces; $\chi(D)=\chi(K)$.
\end{itemize}

It is easy to see these properties fail, even for almost positive diagrams. However, a good successively almost positive diagram has the same properties.

\begin{theorem}
\label{theorem:split-visible}
A good successively $k$-almost positive diagram $D$ represents a split link if and only if $D$ is non-split. 
\end{theorem}

\begin{theorem}
\label{theorem:canonical-surface}
Assume that $D$ is a good successively $k$-almost positive diagram. Then its canonical Seifert surface $S_D$ attains the maximum euler characteristic. 
\end{theorem} 

Besides these properties, it turns out that a good successively almost positive link have more nice properties.

Let $\chi(K)$ be the maximal euler characteristic of Seifert surfaces of $K$.
For the Conway polynomial $\nabla_K(z)$ of $K$, $\max \deg_z \nabla_K(z) = 1-\chi(K)$ holds if $K$ is non-split and positive. Similarly, for the HOMFLY polynomial $P_K(v,z)$ of $K$\footnote{Here we use the convention that the skein relation of the HOMFLY polynomial is $v^{-1}P_{K_+}(v,z)-vP_{K_-}(v,z)=zP_{K_0}(v,z)$.} $\max \deg_z P_K(v,z)= \min \deg_v P_K(v,z)=1-\chi(K)$ holds if $K$ is positive \cite{cr}.

We extend these properties for good successively almost positive links.

\begin{theorem}
\label{theorem:Qhomfibered}
$\max \deg_z \nabla_K(z)=1-\chi(K)$ if $K$ is a non-split good successively almost positive link $K$. 
\end{theorem}

\begin{theorem}
\label{theorem:HOMFLY}
For a good successively almost positive link $K$, $\max \deg_z P_K(v,z) = \min \deg_v P_K(v,z)=1-\chi(K)$.
\end{theorem}

When $K$ is an almost positive link which admits a non-good (successively) $1$-almost positive diagram, the equality $\min \deg_v P_K(v,z)=1-\chi(K)$ was proven in \cite{st} (see Corollary \ref{cor:type-II} for a different proof). Therefore Theorem \ref{theorem:HOMFLY} positively answers \cite[Question 2]{st};

\begin{corollary}
\label{cor:HOMFLY-almost-positive}
$\min \deg_v P_K(v,z)=1-\chi(K)$ if $K$ is almost positive.
\end{corollary}

A \emph{strongly quasipositive link} is another generalization of positive links. 
An $n$-braid is \emph{strongly quasipositive} if it is a product of positive band generators $a_{i,j} = (\sigma_i\sigma_{i+1} \cdots \sigma_{j-2})\sigma_{j-1}(\sigma_i\sigma_{i+1} \cdots \sigma_{j-2})^{-1}$ $(1\leq i<j \leq n)$, where $\sigma_{i}$ is the standard generator of the braid group $B_n$. We say that a link is \emph{strongly quasipositive} if it is represented as the closure of a strongly quasipositive braid. 

There is an equivalent definition using Seifert surfaces; A Seifert surface is \emph{quasipositive} if $S$ is realized as an incompressible subsurface of the fiber surface of a positive torus link. A link is strongly quasipositive if and only if it bounds a quasipositive Seifert surface \cite{ru1}.

In \cite{fll}, Feller-Lewark-Lobb proved that almost positive links are strongly quasipositive. At this moment we do not know whether a $k$-almost positive link is strongly quasipositive or not. However, for a good successively $k$-almost positive link, we see that they are strongly quasipositive.

\begin{theorem}
\label{theorem:SQP}
A good successively almost positive link is strongly quasipositive. Indeed, a canonical Seifert surface $S_D$ of a good successively almost positive diagram $D$ is quasipositive. 
\end{theorem}

Let $s(K),\tau(K)$ be the Rasmussen invariant \cite{ra} and Heeggard Floer tau-invariant \cite{os} of a knot\footnote{A similar result holds for the link case by using appropriate generalizations of the Rasmussen and the tau invariants for links.} $K$, and let $g_4(K)$ be the slice genus of $K$. Since strongly quasipositive knot $K$ satisfies $s(K)=2\tau(K)=2g_4(K)=2g_3(K)$ \cite[Theorem 4]{li}, \cite[Proposition 1.7]{sh}, we conclude the following.

\begin{corollary}
If $K$ is a good successively almost positive knot, then $s(K)=2\tau(K)=2g_{4}(K)=2g(K)$.
\end{corollary}

\subsection{Signature and concordance finiteness}

It is known that the signature of a non-trivial (almost) positive link is always strictly negative; $\sigma(K) < 0$ \cite{pt}. Indeed, for a positive diagram $D$, a stronger inequality 
\begin{equation}
\label{eqn:bdl} \sigma(K) \leq \frac{1}{24}(\chi(D)-1) \; \; \left( = \frac{1}{24}(\chi(K)-1) \right)
\end{equation}
was proven in \cite[Theorem 1.2]{bdl}. Since positive-to-negative crossing change increases the signature by at most two, (\ref{eqn:bdl}) implies 
\begin{equation}
\label{eqn:bdl2} \sigma(K) \leq \frac{1}{24}(\chi(D)-1)+2c_-(D)
\end{equation} for general link diagram $D$.

In this opportunity, we present an improvement of (\ref{eqn:bdl2}).

\begin{theorem}
\label{theorem:signature-improved}
Let $D$ be a reduced diagram of a non-trivial link $K$. Then 
\[ \sigma(K) \leq \frac{1}{12}(\chi(D)-1)+\frac{4}{3}c_-(D) -\frac{1}{2} \leq \frac{1}{12}(\chi(K)-1)+\frac{4}{3}c_-(D)-\frac{1}{2}\]
\end{theorem}

This inequality is interesting and useful in its own right. For positive link case, this gives the following improvement of (\ref{eqn:bdl}).

\begin{corollary}
\label{cor:sig-vs-euler}
If $K$ is a non-trivial positive link, 
\[ \sigma(K) \leq \frac{1}{12}(\chi(K)-1)-\frac{1}{2}. \]
\end{corollary}

In \cite[Theorem 1.1]{bdl}, as an application of signature estimate they showed that every topological knot concordance class contains finitely many positive knots. Theorem \ref{theorem:signature-improved} leads to a similar finiteness result for successively $k$-almost positive knots under an additional restriction.
Let $g_c^{top}(\mathcal{K})= \min \{g(K) \: | \: K \in \mathcal{K}\}$ be the \emph{topological concordance genus} of a topological concordance class $\mathcal{K}$.
 
\begin{theorem}
\label{cor:finite-concordance}
For any $0<d<\frac{1}{8}$, every topological knot concordance class $\mathcal{K}$ contains only finitely many successively $k$-almost positive knots such that $k \leq d g^{top}_c(\mathcal{K})$.
\end{theorem}

In particular, this extends the concordance finiteness result for almost positive knots, under the assumption that its topological concordance genus is sufficiently large.

\begin{corollary}
\label{cor:finite-almost-positive}
A topological knot concordance class $\mathcal{K}$ contains only finitely many almost positive knots if $g^{top}_c(\mathcal{K})\geq 8$.
\end{corollary}

\section*{Acknowledgement}
The author would like to thank Kimihiko Motegi and Masakazu Teragaito for stimulating conversations and discussions. A notion of successively positive diagram appeared during a joint work \cite{imt}.
The author has been partially supported by JSPS KAKENHI Grant Number 19K03490,	21H04428.

\section{Proof of properties}

In this section we present a proof of properties of (good) successively almost positive links.

\subsection{Positive skein resolution tree for successively almost positive diagrams}
\label{section:skein}

\begin{definition}[Skein resolution tree]
For a knot diagram $D$, a \emph{skein resolution tree} is a rooted binary tree $T$  such that
\begin{itemize}
\item[(i)] Each node $c$ is labelled by a knot diagram $D_c$.
\item[(ii)] The root is labelled by $D$.
\item[(iii)] At each non-terminal node $c$ with its children $c',c''$, their labellings $(D_c, D_{c'},D_{c''})$ forms a skein triple $(D_+,D_-,D_0)$ or $(D_-,D_+,D_0)$
\item[(iv)] At each terminal node $c$, $D_c$ represents the trivial link.
\end{itemize}

We say that a skein resolution tree $T$ is \emph{positive} if instead of (iii), the stronger property holds.
\begin{itemize}
\item[(iii+)] At each non-terminal node $c$ with its children $c',c''$, their labellings $D_c, D_{c'},D_{c''}$ forms a skein triple $(D_+,D_-,D_0)$.
\end{itemize}

\end{definition}

By the skein relation, one can compute the Conway, the Jones and the HOMFLY polynomial by the skein resolution tree.

There are various ways to construct a skein resolution tree from a diagram. Here we use the one given in \cite[Theorem 2]{cr}.

A \emph{based link diagram} $D$ is a link diagram $D$ such that
\begin{itemize}
\item The ordering (numbering) of component of $D$ is given; $D=D_1 \cup D_2 \cup \cdots \cup D_n$.
\item For each component $D_i$, the base point $\ast_i \in D_i$ is given.
\end{itemize}

By suitably assigning the $z$-coordinates, we view a based link diagram $D$ as a collection of curves $\gamma_i:[0,1] \rightarrow \R^{3}$ such that $\gamma_i(0)=\gamma_i(1)=\ast_i$.

\begin{definition} 
A based link diagram $D$ is \emph{descending} if 
\begin{enumerate}
\item The $z$-coordinate of $\gamma_i$ is monotone decreasing, except near $t=1$.
\item When $j<i$, the $j$-th component of $D$ lies entirely above of the $i$-th component of $D$.
\end{enumerate} 
\end{definition}
Obviously, a descending diagram represent the unlink.

For a given based link diagram $D$, by trying to make $D$ descending, one can construct skein resolution tree as follows (see \cite[Theorem 2]{cr}, for details).

We start at the base point $\ast_1$ of the 1st component $D_1$. We walk along the diagram until we encounter  the \emph{first non-descending crossing} $c$, the crossing where the descending condition fails for the first time. Namely, we encounter the crossing $c$ for the first time but we pass the crossing $c$ as the underarc. When we do not encounter such a crossing, then we turn our attention to the 2nd component $D_2$, and continue to do the same procedure.

We denote by $d(D)$ the number of crossings that appeared before arriving to the first non-descending crossing. If we cannot find the first non-descending crossing, then $D$ is already descending, and in this case we define $d(D)=c(D)$.
 
Assume that the first non-descending crossing $c$ is the crossing between the components $D_i$ and $D_j$ $(i \leq j)$. We apply the skein resolution at the crossing $c$. Let $D'$ be the diagram obtained by the crossing change at $c$, and let $D''$ be the diagram obtained by the smoothing at $c$. 

The diagram $D'$ is naturally regarded as a based link diagram.
We view the diagram $D''$ as a based diagram as follows:
\begin{itemize}
\item When the smoothing merges two different components (i.e. $i \neq j$), we simply forget the base point and merged component $D_j$.
\item When the smoothing splits the component $D_i$ into two components $D^{(1)}_i$ and $D^{(2)}_i$ (i.e. $i = j$), we take the component that contains $\ast_i$ as the $i$th component. We take $\ast_i$ as its base point.
We take the other component $D^{(2)}$ as $D_{i'}$ where $i'$ is taken so that $i<i'$, and we take a base point near the crossing $c$.
\end{itemize}

By induction on $(c(D)-d(D),c(D))$, one can check that this procedure terminates in finitely many steps hence we get a skein resolution tree. We call this \emph{the descending skein resolution tree} of the based diagram $D$ (see Figure \ref{fig:skein-tree} for example).

\begin{figure}[htbp]
\includegraphics*[width=60mm]{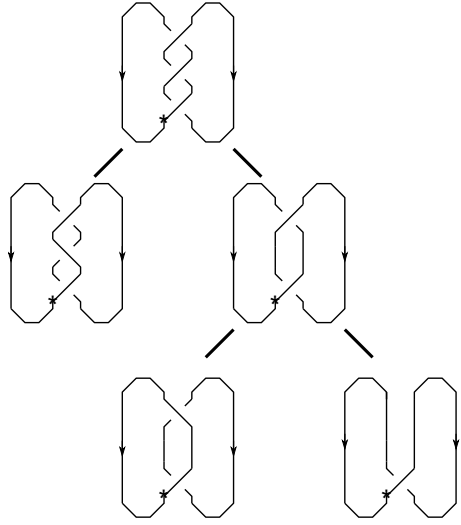}
\begin{picture}(0,0)
\end{picture}
\caption{Descending skein resolution tree of the trefoil knot diagram} 
\label{fig:skein-tree}
\end{figure} 

For a successively $k$-almost positive diagram $D$, we take a base point $\ast$ on the overarc of the first negative crossing poin. We call $\ast$ the \emph{standard base point} of $D$.
We view a successively $k$-almost positive diagram $D$ as a based link diagram by taking the ordering of components so that the component having the standard base point $\ast$ as $D_1$. We take the standard base point as the base point of $D_1$, and the rest of the choices (orderings and the base points of other components) are arbitrary.

\begin{proposition}
\label{prop:positive-resolution}
The descending skein resolution tree of succesively $k$-almost positive diagram $D$ is positive.
\end{proposition}
\begin{proof}
By definition of successively $k$-almost positive diagram $D$, the diagram $D$ is descending at the first $k$ negative crossings. Thus in the construction of the descending skein resolution tree, all the skein resolutions happen at the positive crossings.
\end{proof}

Theorem \ref{theorem:signature} and Theorem \ref{theorem:conway-non-negative} follow from Proposition \ref{prop:positive-resolution} and the following simple observations.
\begin{lemma}
If a diagram $D$ admits a positive skein resolution tree, then
\begin{itemize}
\item[(i)] $\nabla_K(z)$ is non-negative.
\item[(ii)] $\sigma_{\omega}(K)\leq 0$
\end{itemize}
\end{lemma} 
\begin{proof}
(i) follows from the skein relation of the Conway polynomial
\[ \nabla_{K_+}(z) =  \nabla_{K_-}(z)+ z\nabla_{K_0}(z)\]
and the observation that for an $n$-component unlink $U_n$, $\nabla_{U_n} = \begin{cases} 1 & (n=1) \\ 0 & (otherwise) \end{cases}$ is a non-negative polynomial .

(ii) follows from the property of the Levine-Tristram signature
\[ \sigma_{\omega}(K_-)\geq \sigma_{\omega}(K_+) \]
and the observation that the existence of positive skein resolution tree implies that one can make $K$ unlink only by crossing changes at positive crossings.
\end{proof}

\subsection{The canonical Seifert surface of a good successively almost positive diagram}

Theorem \ref{theorem:SQP} and Theorem \ref{theorem:canonical-surface} follows from a characterization of quasipositive canonical Seifert surface \cite[Theorem B]{fll}. Here we give a direct proof based on Murasgui-Prytzcki's operation reducing the number of Seifert circle \cite{mp}. This tells us how to get a strongly quasipositive braid representative of $K$.

A crossing $c$ of $D$ is called \emph{independent} if $c$ connects two Seifert circles $s,s'$ then there are no other crossing connecting $s$ and $s'$.

For an independent crossing $c$ connecting Seifert circles $s,s'$, let $s_1,\ldots,s_{k}$ be the Seifert circles of $D$ other than $s'$ connected to $s$ by a crossing. Let $c_1,\ldots,c_n$ be the crossings that connects $s$ and $s_i$ ($i=1,\ldots,k$).

We move the underarc of the crossing $c$ across one of the Seifert circles $s$, swallowing all the Seifert circles $s_1,\ldots, s_n$ and the crossings $c_1,\ldots,c_n$ to get a new diagram $D'$. We call this operation the \emph{Murasugi-Prytzcki's move} (\emph{MP-move}, in short) at $c$ (see Figure \ref{fig:MP-move}). 
The MP-move reduces the number of Seifert circles of $D$ by one.

\begin{figure}[htbp]
\includegraphics*[width=60mm]{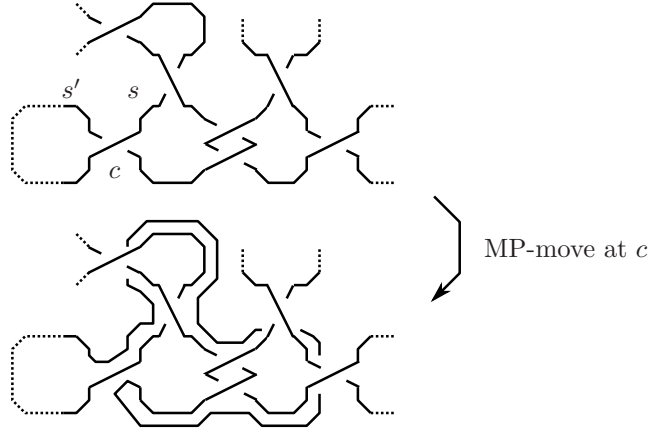}
\begin{picture}(0,0)
\put(-137,95) {$c$}
\put(-155,125) {$s'$}
\put(-130,125) {$s$}
\put(5,65) {MP-move at $c$}
\end{picture}
\caption{Murasugi-Prytzcki's move at the independent crossing $c$} 
\label{fig:MP-move}
\end{figure} 

We view the MP-move as an isotopy of canonical Seifert surface of $D$ that flips the twisted bands corresponding to $c$ and disk bounded by $s$ to merge the disk bounded by $s$ and $s'$ into a single disk. After the MP-move, the Seifert surface still has a disk-and-twisted-band decomposition structure (see Figure \ref{fig:MP-move-surface}).
 
\begin{figure}[htbp]
\includegraphics*[width=60mm]{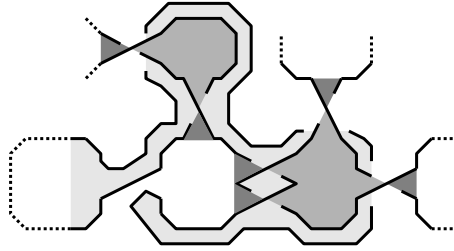}
\begin{picture}(0,0)
\end{picture}
\caption{Murasugi-Prytzcki's move viewed as an isotopy of canonical Seifert surface; the Seifert surface inherits a disk-and-twisted-band decomposition.} 
\label{fig:MP-move-surface}
\end{figure} 
\begin{proof}[Proof of Theorem \ref{theorem:canonical-surface} and Theorem \ref{theorem:SQP}]
When $D$ is good successively $k$-almost positive diagram, we can apply the MP-move $k$ times to get a diagram $D'$ (see Figure \ref{fig:MP-move2}).

Then $s(D') = s(D)-k$ and $w(D')=w(D)+k$.
By the Bennequin's inequality $-s(D')+w(D') \leq -\chi(K)$ hence we get
\begin{align*}
-s(D)+w(D)+2k & = -s(D')+ w(D') \\
& \leq -\chi(K)\\
& \leq -\chi(S_D)= -s(D)+c(D) = -s(D)+w(D)+2k.
\end{align*}
Therefore the canonical Seifert surface $S_D$ attains the maximal euler characteristic.

Moreover, as in the single MP-move case, the sequence of MP-moves also can be seen as an isotopy of canonical Seifert surface of $D$; we see that the canonical Seifert circle $S_D$ can be understood as a surface made of $s(D')=s(D)-k$ disks and $w(D')=w(D)+k$ positively twisted bands.
This disk-and-twisted-band decomposition satisfies a certain nice condition, which we called \emph{quasi-canonical Seifert surface} in \cite{hik}. As we have discussed in \cite[Section 6]{hik}, one can further deform the diagram into a closed braid diagram, preserving the disk and twisted band decomposition structure of $S_D$. Since every twisted band has twisted in a positive direction, the closed braid obtained from $S_D$ is the closure strongly quasipositive braid (see \cite[Theorem 6.4]{hik}.
\end{proof}

\begin{figure}[htbp]
\includegraphics*[width=85mm]{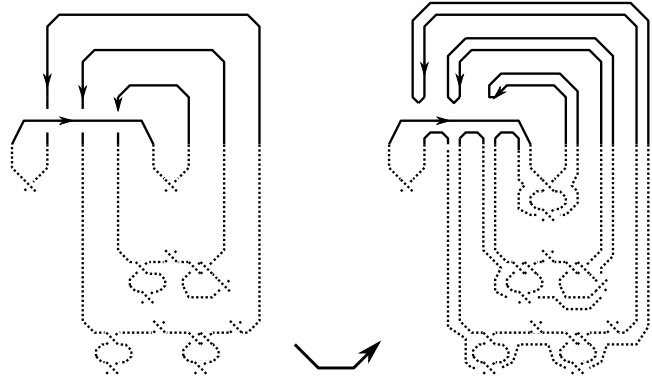}
\begin{picture}(0,0)
\end{picture}
\caption{Applying Murasgui-Prytzcki's move $k$ times for good successively $k$-almost positive diagram ($k=3$ case illustration).}
\label{fig:MP-move2}
\end{figure}

\subsection{Link polynomials and $\chi(K)$}

Theorem \ref{theorem:canonical-surface} allows us to relate $\chi(K)$ and the knot polynomials.

\begin{proof}[Proof of Theorem \ref{theorem:split-visible} and Theorem \ref{theorem:Qhomfibered}]

First we note that Theorem \ref{theorem:split-visible} follows from Theorem \ref{theorem:Qhomfibered}, since $\nabla_K(z)\neq 0$ implies that $K$ is non-split.
 
We prove the assumption by induction of the number $p$ of positive crossing of $D$.
When $p=0$, the assertion is clear. Let us consider the skein triple $(D=D_+,D_-,D_0)$ at the root of the descending skein tree (the resolution at the first non-descending crossing). Note that $D_0$ is a good successively $k$-almost positive diagram.

When $D_0$ is non-split, by induction and Theorem \ref{theorem:canonical-surface}, $\deg_z \nabla_{K_0}(z)=1-\chi(K_0)= 1- s(D) + (c(D)-1)= - \chi(K)$. 
On the other hand, 
\[ \nabla_{K_-}(z) \leq 1-\chi(K_-) \leq 1-\chi(D_-) = 1-\chi(D_+)=1-\chi(K).\]
Since $\nabla_K(z)=\nabla_{K_-}(z)+z\nabla_{K_0}(z)$ and both $\nabla_{K_-}(z)$ and $z\nabla_{K_0}(z)$ are non-negative, $\deg \nabla_{K}(z)= 1-\chi(K)$ as desired.  

When $D_0$ is split, the crossing $c$ should be a nugatory crossing so $K$ is represented by a non-split good successively $k$-almost positive diagram having $(p-1)$ positive crossings. Therefore by induction $\nabla_{K}(z)=1-\chi(K)$.

\end{proof}

\begin{proof}[Proof of Theorem \ref{theorem:HOMFLY}]
Assume that $K$ admits a good successively $k$-almost positive diagram $D$.
We prove the theorem by induction on $k$.

We consider the skein triple $(D_+,D=D_-,D_0)$ obtained by the first negative crossing $D$. 
Then $D_+$ and $D_0$ are good successively $(k-1)$-almost positive diagram hence by Theorem \ref{theorem:canonical-surface}
\[ \chi(K_+)=\chi(K), \chi(K_0)=\chi(K)-1\]
Therefore by induction
\[\max \deg_z P_{K_+}(v,z) = 1-\chi(K), \max \deg_z P_{K_0}(v,z)=-\chi(K). \]

By the skein relation $P_{K}(v,z) = v^{-2}P_{K_{+}}(v,z) - v^{-1} zP_{K_{0}}(v,z)$ we get 
\[ \max \deg_z P_{K}(v,z) \leq 1-\chi(K).\]
On the other hand,  by Theorem \ref{theorem:Qhomfibered}
\[ 1-\chi(K)=\deg_z \nabla_K(z) = \max \deg_z P_K(v,z)|_{v=1} \leq \max \deg_z P_{K}(v,z) \]
hence we conclude that $\max \deg_z P_{K}(v,z) = 1-\chi(K)$.

Next we show $\min \deg_v P_{K}(v,z) = 1-\chi(K)$. 
By the skein relation (characterization) of the HOMFLY polynomial $P_K(v,-v^{-1}+v)=1$. This implies that $P_K(v,z)$ contains a monomial $v^a z^{b}$ whose coefficient is non-zero and $a-b \leq 0$. Thus 
\[ \min \deg_v P_K(v,z) \leq a \leq b \leq \max \deg_z P_K(v,z). \] 
On the other hand, by the Morton-Franks-William inequality 
\[ \overline{sl}(K) +1 \leq \min \deg_v P_K(v,z) \]
where $\overline{sl}(K)$ is the maximal self-linking number of $K$. Since we have seen that $K$ is strongly quasipositive, $\overline{sl}(K)=-\chi(K)$. Combining the (in)equalities we get
\[ -\chi(K) +1= \overline{sl}(K) +1  \leq \min \deg_v P_K(v,z) \leq \max \deg_z P_K(v,z) = -\chi(K)+1.\]
\end{proof}

\subsection{Signature estimate}

We turn to our attention to Theorem \ref{theorem:signature-improved}.
In large part, our argument goes the same line as Baader-Dehornoy-Liechti's argument \cite{bdl} adapted so that it can be applied to general link diagrams with slight improvements.

To treat the signature, we use Gordon-Litherland's theorem \cite{gl}. For a (possibly non-orientable) spanning $F$ of a link $K$, let $\langle \; , \; \rangle_F :H_1(F) \times H_1(F) \rightarrow \Z$ be the \emph{Gordon-Litherland pairing} of $F$; for $a,b \in H_1(F)$, let $\alpha,\beta$ be curves on $S$ that represent $a$, $b$, and let $p_F: \nu(F) \rightarrow F$ be the unit normal bundle of $F$. The Gordon-Litherland pairing of $a$ and $b$ is defined by $\langle a , b \rangle_F = lk(\alpha, p_S^{-1}(\beta))$, where $lk$ denotes the linking number.

When $K$ is an $\ell$-component link $K_1 \cup \cdots \cup K_{\ell}$, the Gordon-Litherland theorem states
\begin{equation}
\label{theorem:GL} \sigma(K)= \sigma(F) + \frac{1}{2}e(F,K).
\end{equation}
Here $\sigma(F)$ is the signature of the Gordon-Litherland pairing of $F$, and 
\[ e(F,K)= - \frac{1}{2}\langle [K] , [K] \rangle_F - \sum_{1\leq i<j \leq \ell}lk(K_i,K_j). \]

For a link diagram $D$, we fix one of its checkerboard coloring, and let $B$ and $W$ be the black and white checkerboard surface of $D$. We say that a crossing $c$ of $D$ is \emph{of type a} (resp. \emph{of type b}) if, when we put the overarc so that it is an horizontal line, the upper right-hand side and the lower left-handed side (resp. the lower right-hand side and the upper left-handed side) are colored by black (see Figure \ref{fig:type-crossings}). Similarly, we say that a crossing $c$ is \emph{of type I} (resp. \emph{of type II}) if the black region is compatible (resp. incompatible) with the orientation of the diagram (see Figure \ref{fig:type-crossings}). In the definition of type a/b, the orientation of $D$ is irrelevant whereas the definition of type I/II, the over-under information is irrelevant.

We say that a crossing $c$ is \emph{of type $Ia$}, for example, if $c$ is both of type I and of type a. We put $c_{Ia},c_{Ib},c_{IIa},c_{IIb}$ the number of crossings of type $Ia,Ib,IIa,IIb$, respectively. Note that positive (resp. negative) crossing is either of type Ib or IIa (resp. Ia or IIb) so
\begin{equation}
\label{eqn:ab-pm} c_+(D)= c_{Ib} + c_{IIa}, \quad c_{-}(D)=c_{Ia}+c_{IIb}.
\end{equation}

\begin{figure}[htbp]
\includegraphics*[width=90mm]{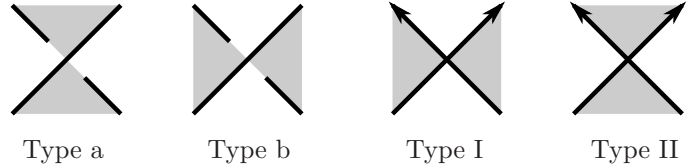}
\begin{picture}(0,0)
\put(-255,-15) {Type a}
\put(-185,-15) {Type b}
\put(-110,-15) {Type I}
\put(-40,-15) {Type II}
\end{picture}
\medskip\medskip\medskip
\caption{Types of crossings with respect to the checkerboard coloring}
\label{fig:type-crossings}
\end{figure}

By the definition of the Gordon-Litherland pairing,
\[ \frac{1}{2}e(B,K) = c_{IIb} - c_{IIa}, \quad \frac{1}{2}e(W,K) = c_{Ia} - c_{Ib}.\]
Thus by (\ref{theorem:GL}) and (\ref{eqn:ab-pm})
\begin{equation}
\label{eqn:sgn-part-1}
 2\sigma(K) = \sigma(B) +\sigma(W) - c_+(D) + c_-(D) 
\end{equation}

We give an estimation of the signature of Gordon-Litherland pairings. Let $\mathcal{R}(W)$ and $\mathcal{R}(B)$ be the set of white and black regions, respectively. For a white region $R \in \mathcal{R}(W)$, we associate a simple closed curve $\gamma_R$ on $B$ which is a mild perturbation of the boundary of $R$ (see Figure \ref{fig:gamma_R}).  

We say that the region $R$ is \emph{of type $(\alpha,\beta)$} if $R$ contains $\alpha$ type a crossings and $\beta$ type b crossings as its corners. By definition, 
\[ \langle [\gamma_R], [\gamma_R] \rangle_B = \alpha-\beta \]
when $R$ is of type $(\alpha,\beta)$.

\begin{figure}[htbp]
\includegraphics*[width=30mm]{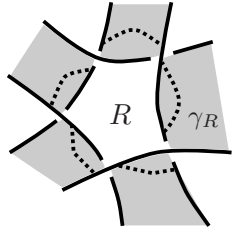}
\begin{picture}(0,0)
\put(-50,40){\large $R$}
\put(-20,40){$\gamma_R$}
\end{picture}
\caption{Curve $\gamma_R$ for a white region $R$}
\label{fig:gamma_R}
\end{figure} 

For a black region $R \in \mathcal{R}(B)$, the curve $\gamma_R$ on $W$ and a notion of type $(\alpha,\beta)$ are defined similarly and the Gordon-Litherland pairing is given by
\[ \langle [\gamma_R], [\gamma_R] \rangle_W = \beta-\alpha \]
when $R$ is of type $(\alpha,\beta)$.

\begin{proof}[Proof of Theorem \ref{theorem:signature-improved}]
Since $D$ is reduced, there are no regions of type $(1,0)$ or $(0,1)$.
Moreover, since we assume that $K$ is non-trivial, $\# \mathcal{R}(D),\#\mathcal{B}(D) \geq 2$.

Let $\Gamma$ be the planar graph whose vertices $\mathcal{V}(\Gamma)$ are white regions of type $(\alpha,\beta)$ with $\alpha-\beta \leq 0$, and two vertices $R,R'$ are connected by an edge if they share a corner. By Appel-Haken's four-color theorem\footnote{It is interesting to find an argument that avoids to use the four-color theorem.} \cite{ah} there is a 4-coloring map $col:V(\Gamma) \rightarrow \{1,2,3,4\}$; if two vertices $v,v'$ are connected by an edge, then $col(v) \neq col(v')$. Let $\mathcal{V}' = col^{-1}(\{1,2\})$. 
Let $\Gamma'$ be the subgraph of $\Gamma$ whose vertices are $\mathcal{V}'$, and let $V_B$ be the subspace of $H_1(B)$ generated by $[\gamma_R]$ for $ R \in \mathcal{V}'$. 
With no loss of generality, we assume that $\#\mathcal{V}' \geq \frac{1}{2} \#\mathcal{V}(\Gamma)$.

Since $\Gamma'$ is bipartite, the restriction of Gordon-Litherland pairing $\langle \; , \; \rangle_F$ on $V_B$ is of the form $\begin{pmatrix}D_1 & X^{T} \\ X & D_2 \end{pmatrix}$ where $D_1,D_2$ are diagonal matrices with non-positive diagonals, hence the Gordon-Litherland pairing is non-positive on $V_B$.

First we assume that $\mathcal{V}'$ is not equal to the whole $\mathcal{R}(W)$. In this case $\{ [\gamma_R] \: |\: R \in \mathcal{V}'\}$ is a basis of $V_B$ so $\dim V_B = \# \mathcal{V}'$. Let $\gamma^{W}_{>0}$ be the number of white region $R$ such that $\langle [\gamma_R], [\gamma_R] \rangle_B>0$. Then
\[ \dim V_B = \# \mathcal{V}' \geq \frac{1}{2}\# \mathcal{V}(\Gamma) = \frac{1}{2}(\# \mathcal{R}(W) - \gamma^{W}_{>0}) \]

On the other hand, when $\mathcal{V}'$ coincides with $\mathcal{R}(W)$, then $\dim V_{B}= \dim H_1(B)=\# \mathcal{R}(W)-1$. Since $\mathcal{R}(W)\geq 2$, we have the same lower bound of $\dim V_B$;
\[ \dim V_B = \# \mathcal{R}(W)-1 \geq \frac{1}{2} \# \mathcal{R}(W) \geq \frac{1}{2}(\# \mathcal{R}(W) - \gamma^{W}_{>0})  \]

Therefore we get an upper bound of the signature of the Gordon-Litherlanf pairing of $B$\footnote{This is a point where a minor improvement (constant $-\frac{1}{2}$ in the conclusion) appears.}.
\begin{align*}
\sigma(B) &\leq \dim H_1(B) - \dim V_B = (\# \mathcal{R}(W)-1)-\dim V_B\\
& \leq \frac{1}{2}\# \mathcal{R}(W) -1 + \frac{1}{2}\gamma^{W}_{>0} 
\end{align*}

By a parallel argument for the white surface $W$, we get a similar estimate
\[ \sigma(W) \leq \frac{1}{2} \# \mathcal{R}(B)-1 + \frac{1}{2}\gamma^{B}_{>0} \]
where $\gamma^{B}_{>0}$ is the number of a black region $R$ such that $\langle \gamma_R, \gamma_R\rangle_W>0$.

Since $\# \mathcal{R}(W)+\# \mathcal{R}(B)-2 = c(D)$, by (\ref{eqn:sgn-part-1}) we get
\begin{equation}
\label{eqn:sigma-gamma}
2\sigma(K) \leq -\frac{1}{2}c(D) +2c_{-}(D)+ \frac{1}{2} \gamma_{>0}^{B} + \frac{1}{2} \gamma_{>0}^{W}-1
\end{equation}

It remains to estimate $\gamma_{>0}^{B}$ and $\gamma_{>0}^{W}$.
Let $\gamma^{W}(\alpha,\beta)$ be the number of white regions of type $(\alpha,\beta)$. By definition of $\gamma^{W}_{>0}$,
\[ \gamma^{W}_{>0} = \sum_{\substack{\alpha>\beta\geq 0 \\ \alpha\geq 2}} \gamma^W(\alpha,\beta). \]
By counting the number of the crossings of type $a$ that appear as a corner of white regions, we get
\begin{align*}
2(c_{Ia}+c_{IIa}) &= \sum_{\alpha,\beta\geq 0} \alpha \gamma^{W}(\alpha,\beta) \\
& \geq 2 \gamma^{W}(2,0) +2 \gamma^{W}(2,1) + 3\sum_{\substack{\alpha>\beta\geq 0 \\ \alpha\geq 3}} \gamma^W(\alpha,\beta) \\
& = -\gamma^{W}(2,0) - \gamma^{W}(2,1) + 3\sum_{\substack{\alpha>\beta\geq 0 \\ \alpha\geq 2}} \gamma^W(\alpha,\beta)\\
& = -\gamma^{W}(2,0) - \gamma^{W}(2,1) + 3\gamma^{W}_{>0}.
\end{align*}
Thus we conclude
\begin{equation}
\label{eqn:gamma_0}
\gamma^{W}_{>0} \leq \frac{2}{3}(c_{Ia}+c_{IIa})+\frac{1}{3}\gamma^{W}(2,0) +\frac{1}{3}\gamma^{W}(2,1) 
\end{equation}

Recall that a positive crossing appears as either of type Ib or of type IIa. If the corner of a white region $R$ is a crossing of type IIa, then the orientation of link and the orientation of the white region switch. Thus if all the corners of $R$ are positive crossings, then the region $R$ must be of type $R(\alpha,2\beta)$.
In particular, if a white region $R$ is of type $(2,1)$, at least one of its corner is a negative crossing. Similarly, if a white region $R$ is of type $(2,0)$, its corners are either both positive, or, both negative.

Let $\gamma^{W}_{\pm}(2,0)$ be the number of white regions of type $(2,0)$ whose corners are positive and negative crossings, respectively. 
By counting the number of negative crossings that appear as a corner of type $(2,1)$ regions or type $(2,0)$ regions we get
\[ \gamma^{W}(2,1) + 2\gamma^{W}_{-}(2,0) \leq 2c_{-}. \]
Thus
\[ \gamma^W(2,1) + \gamma^W_{-}(2,0) \leq 2c_{-}. \]

Let $s_W(D)$ be the number of  Seifert circles of $D$ which are the boundary of white bigon. When both corners of a white region $R$ of type $(2,0)$ are positive, then the boundary of $R$ forms a Seifert circle of $D$ so\footnote{A substantial improvement appears at this point; In \cite{bdl} they used an upper bound of $\gamma^{W}_+(2,0)$ in terms of the crossing numbers, as we do for $\gamma^{W}_-(2,0)$, instead of the number of Seifert circles.} 
\[ \gamma^{W}_{+}(2,0) \leq s_W(D).\]
By (\ref{eqn:gamma_0}) we conclude
\begin{equation}
\label{eqn:gammaW}
\gamma^{W}_{>0} \leq \frac{2}{3}(c_{Ia}+c_{IIa})+\frac{2}{3}c_- +s_W(D).
\end{equation}

By the same argument for the white surface, we get a similar inequality
\begin{equation}
\label{eqn:gammaB}
\gamma^{B}_{>0} \leq \frac{2}{3}(c_{IIa}+c_{Ib})+\frac{2}{3}c_- +s_B(D).
\end{equation}
where $s_B(D)$ is the number of Seifert circles which are the boundary of black bigon.

Since a Seifert circle cannot be the boundary of white bigon and black bigon at the same time,
\[ s_W(D)+s_B(D) \leq s(D). \]
The equality happens only if all the Seifert circles are bigon. Thus the equality occurs only if $D$ is the standard torus $(2,2n)$ link diagram (with opposite orientation, so that they bound an annulus). In this case, the asserted inequality of signature is obvious, so in the following we can assume that a bit stronger inequality 
\begin{equation}
\label{eqn:s-circle} s_W(D)+s_B(D) \leq s(D)-1. 
\end{equation}

By (\ref{eqn:gammaW}), (\ref{eqn:gammaB}), and (\ref{eqn:s-circle})
\begin{align*}
\gamma^{B}_{>0} + \gamma^{W}_{>0} &\leq  \frac{2}{3}(c_{Ia}+c_{IIa}+c_{Ib}+c_{IIb}) + \frac{1}{3}(s(D)-1) + \frac{4}{3}c_{-}(D) \\
& =\frac{2}{3}c(D) + \frac{1}{3}(s(D) -1) + \frac{4}{3}c_{-}(D).
\end{align*}

Therefore by (\ref{eqn:sigma-gamma}) we conclude 
\begin{align*}
2\sigma(K) & \leq -\frac{1}{2}c(D) +2c_-(D) + \frac{1}{2} \gamma_{>0}^{B} + \frac{1}{2} \gamma_{>0}^{W}-1  \\
& \leq -\frac{1}{2}c(D) +2c_-(D) + \frac{1}{3}c(D) +\frac{1}{6}(s(D)-1) + \frac{2}{3}c_{-}(D)-1 \\
 & = \frac{1}{6}(s(D)-c(D)-1) + \frac{8}{3}c_-(D)-1.
\end{align*}

\end{proof}

Once a signature estimate from canonical Seifert surface is established, Theorem \ref{cor:finite-concordance} (the concordance finiteness) can be proved by almost the same argument as \cite[Theorem 1.1]{bdl}.
 
Recall the Levine-Tristram signature $\sigma_{\omega}(K)$ of a link $K$ is topological concordance invariant whenever $\Delta_\omega(K)\neq 0$ so if we take $\omega$ as non-algebraic number, then $\sigma_{\omega}$ is always a topological concordance invariant.

\begin{proof}[Proof of Theorem \ref{cor:finite-concordance}]

Assume, to the contrary that the topological concordance class $\mathcal{K}$ contains infinitely many successively positive knots $\{K_i\}$.

Let $D_i$ be a successively $k_i$-almost positive diagram of $K_i$. By assumption,
$k_i \leq d g^{top}_c(K) \leq d g(K_i) \leq d g(D_i)$. Therefore by Theorem  \ref{theorem:signature-improved} 
\begin{align*}
\sigma(\mathcal{K}) = \sigma(K_i) 
& \leq  -\frac{1}{6}g(D_i)+\frac{4}{3}k_i -\frac{1}{2} \leq \left( -\frac{1}{6} + \frac{4}{3}d \right)g(D_i).
\end{align*}
Since $d<\frac{1}{8}$, $(-\frac{1}{6} +\frac{4}{3}d)<0$. Therefore 
\[ g(D_i) \leq \frac{6}{8d-1}\sigma(\mathcal{K}). \]

Since the canonical genus of the diagrams $\{D_i\}$ are bounded above, there is a finite set of diagrams $\mathcal{D}$ such that each $D_i$ is obtained from one of a diagram $D'_{i} \in \mathcal{D}$ by inserting full twists $N$ times  (often called the \emph{$\overline{t}_N$-move}) for some $N$, at appropriate crossings of $D'_{i}$ \cite[Theorem 3.1]{st1} (see Figure \ref{fig:t_N-move}).
Since every $D_i$ is successively almost positive, we may assume that the base diagram $D'_i \in \mathcal{D}$ is successively positive, and that $D_i$ is obtained by inserting \emph{positive} twists.

\begin{figure}[htbp]
\includegraphics*[width=60mm]{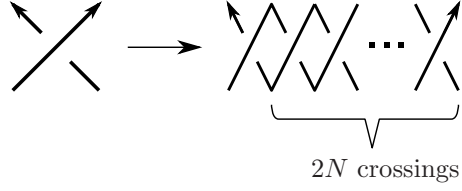}
\begin{picture}(0,0)
\put(-60,-10){$2N$ crossings}
\put(-20,40){}
\end{picture}
\medskip
\medskip
\caption{Inserting $N$ full twists at the crossing $c$ (often called the $\overline{t}_N$-move)}
\label{fig:t_N-move}
\end{figure}

Due to the finiteness of $\mathcal{D}$, there are only finitely many places to insert full twists. Thus there is a diagram $D_{\sf base} \in \mathcal{D}$ and a crossing $c$ of $D_{\sf base}$ having the following property;
for any $N>0$, there is a knot $K$ in $\{K_i\}$ such that $K$ is obtained from $\mathcal{D}$ by inserting full twist at least $N$-times at $c$ and inserting appropriate (positive) twists at other crossings.

Let $D(N)$ be the diagram obtained from $D_{\sf base}$ by inserting $N$-full twists at the crossing $c$, and let $K(N)$ be the knot represented by $D(N)$.
The above observation says that there exists a successively positive knot $K$ in the concordance class $\mathcal{K}$ such that $K(N)$ is obtained from $K$ by the positive-to-negative crossing changes.
Thus for every $N>0$ we have an inequality
\begin{equation}
\label{eqn:K-N} \sigma_{\omega}(K(N))\geq \sigma_{\omega}(K)=\sigma_{\omega}(\mathcal{K}).
\end{equation}

Let $D_{0}$ be the link diagram obtained by smoothing the crossing $c$ of $D_{\sf base}$ and let $L$ be the link represented by $D_{0}$. Since the canonical Seifert surface $S_N$ of $D(N)$ is obtained from the canonical Seifert surface $S_0$ of $D_0$ by adding an $N$-twisted band, the Seifert matrix $A(N)$ of $S_N$ is of the form
\begin{equation*}
 A(N)= \begin{pmatrix}-N+a & w \\ v & A(S_0) \end{pmatrix} 
\end{equation*}
where $a$ is a constant and $A(S_0)$ is the Seifert matrix of $S_0$.

Take a non-algebraic $1\neq \omega \in \{z \in \C \: | \: |z|=1\}$ sufficiently close to $1$ so that $\sigma_{\omega}(\mathcal{K})=0$ holds. By definition, $\sigma_{\omega}(K(N))$ and $\sigma_{\omega}(L)$ are the signatures of
\begin{align*}
A_{\omega}(N) & = (1-\omega)A(N) + (1-\overline{\omega})A(N)^T \\
 & = \begin{pmatrix} (2-2\mathsf{Re}(\omega))(-N+a) & (1-\omega)w+(1-\overline{\omega})v^{T} \\
(1-\omega)v+(1-\overline{\omega})w^{T} & (1-\omega)A(S_{0}) + (1-\overline{\omega})A(S_{0})^T \end{pmatrix}\\
A_{\omega}(0)&= (1-\omega)A(S_{0}) + (1-\overline{\omega})A(S_{0})^T.
\end{align*}Therefore by the cofactor expansion
\[  \det A_{\omega}(N) = (2-2\mathsf{Re}(\omega))(-N+a)\det A_{\omega}(0)+ C\]
where $C$ is a constant that does not depend on $N$.
Since we have chosen $\omega$ so that it is non-algebraic, $\det A_{\omega}(S_{0}) \neq 0$. Thus when $N$ is sufficiently large, the sign of $\det A_{\omega}(N)$ and $\det A_{\omega}(S_{D_0})$ are different, which means that the matrix $A_{\omega}(N)$ has one more negative eigenvalue than $A_{\omega}(0)$. Thus for sufficiently large $N$
\[ \sigma_{\omega}(K(N))= \sigma_{\omega}(L)-1. \]

On the other hand, since $D_0$ is successively positive, by Theorem \ref{theorem:signature-improved} $\sigma_{\omega}(L) \leq 0$. Therefore by (\ref{eqn:K-N})
\[ \sigma_{\omega}(\mathcal{K}) \leq \sigma_{\omega}(K(N)) =\sigma_{\omega}(L)-1 \leq -1. \]
Since we have chosen $\omega$ so that $\sigma_{\omega}(\mathcal{K})=0$, this is a contradiction.
\end{proof}

\section{Discussions and Questions}

Our argument so far justifies the assertion that a (good) successively $k$-almost positive diagram is natural and more appropriate generalization of positive links than a $k$-almost positive diagram. 
However, since our results also says that the difference between (almost) positive links and (good) successively $k$-almost positive links are subtle to distinguish.

In fact, at this moment we do not know an example of links which are (good) successively almost positive but is not almost positive, although we believe that there are many such links. This is mainly because we have no useful technique to exclude various candidates are indeed not almost positive.

\subsection{Positive v.s. almost positive}

Since positive links and almost positive links already share many properties, distinguishing almost positive links with positive link is not an easy task.

The simplest example of almost positive, but not positive knot is $10_{145}$; $10_{145}$ admits an almost positive diagram of type I. The non-positivity can be detected by Cromwell's theorem \cite{cr} that $c(K) \leq 4g(K)$ if $K$ is positive and its Conway polynomial $\nabla_K(z)$ is monic, or, the property of the HOMFLY and the Kauffman polynomial of positive knot (\ref{eqn:Yokota}) which we discuss in the next section.

\subsection{Almost positive of type I v.s. Almost positive of type II}
\label{sec:ap-I-vs-II}
Our argument so far tells that a good successively almost positive diagram has better properties than a usual almost positive diagram. 

However, we note that it can happen non-good almost positive links share a property of positive links that fails for good almost positive links.

An almost positive diagram $D$ is  
\begin{itemize}
\item[-] \emph{of type I} if $D$ is good successively almost positive; for two Seifert circles $s$ and $s'$ connected by the unique negative crossing $c_-$, there are no other crossings connecting $s$ and $s'$.
\item[-] \emph{of type II} otherwise; there are positive crossings connecting $s$ and $s'$. 
\end{itemize}

An almost positive diagram of type I is nothing but a good successively almost positive diagram. For an almost positive diagram $D$, $S_D$ attains the maximum euler characteristic if and only if $D$ is of type I \cite{st}. In particular, when $D$ is an almost positive diagram of type II, we have $\chi(K)=\chi(D)+2$. The dichotomy plays a fundamental role in a study of almost positive diagrams and links -- the proof of various properties of almost positive links often splits into the analysis of two cases \cite{st,fll}.

Let $D_K(a,z)$ be the Dubrovnik version of the Kauffman polynomial;
$D_K(a,z)=a^{-w(D)}\Lambda_{D}(a,z)$, where $\Lambda_D(a,z)$ is the regular isotopy invariant defined by the skein relations
\[ \Lambda_{ \raisebox{-1mm}{\includegraphics*[width=4mm]{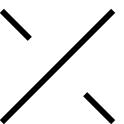}}}(a,z) - \Lambda_{ \raisebox{-1mm}{\includegraphics*[width=4mm]{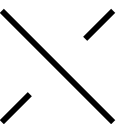}}}(a,z) = z(\Lambda_{ \raisebox{-1mm}{\includegraphics*[width=4mm]{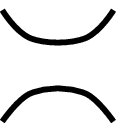}}} -\Lambda_{ \raisebox{-1mm}{\includegraphics*[width=4mm]{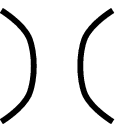}}}), \]

\[  \Lambda_{ \raisebox{-1mm}{\includegraphics*[width=4mm]{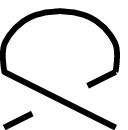}}}(a,z) = a \Lambda_{ \raisebox{-1mm}{\includegraphics*[width=4mm]{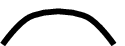}}}(a,z), \Lambda_{ \raisebox{-1mm}{\includegraphics*[width=4mm]{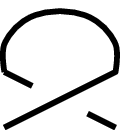}}}(a,z) = a^{-1} \Lambda_{ \raisebox{-1mm}{\includegraphics*[width=4mm]{skeinst.eps}}}(a,z), \Lambda_{ \raisebox{-1mm}{\includegraphics*[width=4mm]{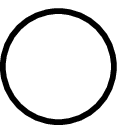}}}(a,z) = 1. \]

We express the Dubrovnik polynomial $D_K(a,z)= \sum_{i} D_K(z;i)a^{i}$. Similarly, we express the HOMFLY polynomial $P_K(v,z)=\sum_{i}P_K(z;i)v^{i}$. In \cite{yo} Yokota showed that when $K$ is positive,
\begin{equation}
\label{eqn:Yokota}
P_K(z;1-\chi(K))=D_K(z;1-\chi(K))\neq 0
\end{equation}
holds. Moreover, this is a non-negative polynomial \cite[Corollary 4.3]{cr}.

At first glance, this coincidence sounds mysterious but it turns out Legendrian link point of view brings illuminating explanation.

For a basis of Legendrian links we refer to \cite{et}.
Let $\mathcal{D}$ be the front diagram of a Legendrian link $K$. The subset $\rho$ of the set of crossings of $\mathcal{D}$ is a \emph{ruling} if the diagram $\mathcal{D}_{\rho}$ obtained from $\mathcal{D}$ by taking the horizontal smoothing $\LCross \to \Smooth$ at each crossing in $\rho$ satisfies the following conditions;
\begin{itemize}
\item Each component of $\mathcal{D}_{\rho}$ is the standard diagram of the Legendrian unknot (i.e it contains two cusps and no crossings).
\item For each $c \in \rho$, let $P$ and $Q$ be the component of $D_{\mathcal{\rho}}$ that contains the smoothed arcs at $c$. 
Then $P$ and $Q$ are differents component of $\mathcal{D}_{\rho}$, and in the vertical slice around $c$, two components $P$ and $Q$ are not nested -- they are aligned in one of the configuration in Figure \ref{fig:ruling} (ii).
\end{itemize}

\begin{figure}[htbp]
\includegraphics*[width=70mm]{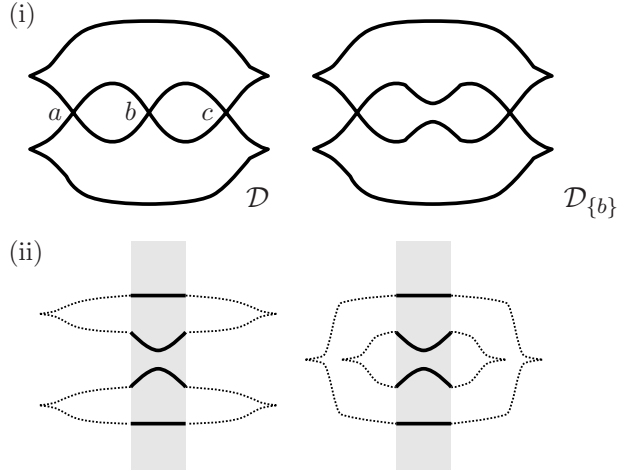}
\begin{picture}(0,0)
\put(-210,170) {(i)}
\put(-210,80) {(ii)}
\put(-195,133) {$a$}
\put(-166,133) {$b$}
\put(-137,133) {$c$}
\put(-120,100) {\large $\mathcal{D}$}
\put(0,100) {\large $\mathcal{D}_{\{b\}}$}
\end{picture}
\caption{(i) Front diagram $\mathcal{D}$ and $\mathcal{D}_{\rho}$ for $\rho=\{b\}$; $\rho$ is not a ruling. (ii) Normality condition for the ruling.}
\label{fig:ruling}
\end{figure}

A ruling is \emph{oriented} if every element of $\rho$ is a positive crossing.
Let $\Gamma(\mathcal{D})$ and $O\Gamma(\mathcal{D})$ be the set of rulings and oriented rulings, respectively.

For a ruling $\rho$ we define 
\[ j(\rho)=\# \rho - l\mbox{-cusp}(\mathcal{D}) +1,\]
 where $l\mbox{-cusp}(\mathcal{D})$ denotes the number of the left cusps of $\mathcal{D}$.
In \cite[Theorem 3.1, Theorem 4.3]{rt}, Rutherford showed that the ruling polynomials, a graded count of the (oriented) rulings, are equal to the coefficient polynomials of Dubrovnik and HOMFLY polynomials, respectively;
\begin{equation}
\label{eqn:ruling-formula}
D_K(z;tb(\mathcal{D})+1) = \sum_{\rho \in \Gamma(\mathcal{D})} z^{j(\rho)}, \quad P_K(z;tb(\mathcal{D})+1) = \sum_{\rho \in O\Gamma(\mathcal{D})} z^{j(\rho)}. 
\end{equation}
Here $tb(\mathcal{D})$ is the Thurston-Benneuqin invariant, given by 
\[ tb(\mathcal{D}) = w(\mathcal{D}) -  l\mbox{-cusp}(\mathcal{D}). \]
In particular, by the HOMFLY/Kauffman bound of the maximum Thurston-Bennequin number $\overline{tb}(K)$
\[ \overline{tb}(K)+1 \leq \min \deg_v P_K(v,z), \quad  \overline{tb}(K)+1 \leq\min \deg_a D_K(v,z), \]
if a front diagram $\mathcal{D}$ admits a ruling, then $\mathcal{D}$ attains the maximum Thurston-Bennequin number $\overline{tb}(K)$ among its topological link types, and 
\[ \overline{tb}(K)+1 = \min \deg_v P_K(v,z) = \min \deg_a D_K(v,z).\]

One can view a positive diagram $D$ as a front diagram $\mathcal{D}$ \cite{ta}; each positive crossing is put so that in the form \Crossing, and each Seifert circle forms a front diagram having exactly two cusps (so $l\mbox{-cusp}(\mathcal{D}) =s(D)$). Then the set of all crossings forms a (oriented) ruling. Thus $\mathcal{D}$ attains the maximum Thurston-Bennequin number. In particular,
\[ \overline{tb}(K) = tb(\mathcal{D}) = c(D)-s(D) = -\chi(D) = -\chi(K).\]
Moreover, since all the crossings of $\mathcal{D}$ are positive, $\Gamma(\mathcal{D}) = O\Gamma(\mathcal{D})$. Thus the ruling polynomial formula (\ref{eqn:ruling-formula}) of the coefficients of the Dubrovnik and the HOMFLY polynomials leads to Yokota's equality (\ref{eqn:Yokota}) and its non-negativity.

A mild generalization of this argument shows that an almost positive link of type II shares the same property;

\begin{theorem}
\label{theorem:almost-type-II}
If $K$ admits an almost positive diagram of type II, then 
\[ P_K(z;1-\chi(K))=D_K(z;1-\chi(K)) \neq 0\]
and it is a non-negative polynomial.
\end{theorem}
\begin{proof}
We view an almost positive diagram of type II as a front diagram $\mathcal{D}$ as shown in Figure \ref{fig:almost-positive-front}, where the box represents a positive diagram part viewed as a front diagram. A ruling of $\mathcal{D}$ cannot contain the negative crossing of $\mathcal{D}$, so $\Gamma(\mathcal{D}) = O\Gamma(\mathcal{D})$. On the other hand, since $D$ is of type II, the Seifert circles $s$ and $s'$ connected by the negative crossing $c_-$ is also connected by a positive crossing. We take such a crossing $c$ so that there are no such crossing between $c$ and $c^{-}$. Let $\rho=C(\mathcal{D}) \setminus \{c_-,c\}$. Then $\rho$ is a ruling so $\Gamma(\mathcal{D}) = O\Gamma(\mathcal{D}) \neq \emptyset$, and
\[ \overline{tb}(K) = tb(\mathcal{D}) = (c(D)-2)-s(D) = -\chi(D)-2 = -\chi(K).\]
Thus by (\ref{eqn:ruling-formula}) we get the desired identity.
\end{proof} 

\begin{corollary}\label{cor:type-II}
If $K$ is almost positive of type II, 
\[ \min \deg_a D_K(z,a)=\max \deg_z P_K(v,z)=1-\chi(K).\]
\end{corollary}

\begin{figure}[htbp]
\includegraphics*[width=100mm]{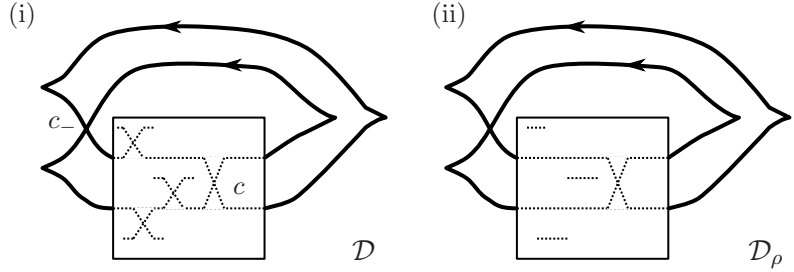}
\begin{picture}(0,0)
\put(-300,90) {(i)}
\put(-140,90) {(ii)}
\put(-285,50) {$c_-$}
\put(-215,25) {$c$}
\put(-170,0) {\large $\mathcal{D}$}
\put(-20,0) {\large $\mathcal{D}_{\rho}$}
\end{picture}
\caption{(i) Front diagram of almost positive diagram of type II; the box represent a front diagram consisting of positive crossings. (ii) Non-trivial ruling $\rho=C(\mathcal{D}) \setminus \{c_-,c\}$ and its resolution $\mathcal{D}_{\rho}$ }
\label{fig:almost-positive-front}
\end{figure} 

Since there are almost positive knot of type I that fails to have the property stated in Theorem \ref{theorem:almost-type-II} (for example $10_{145}$), we have the following, which was mentioned in \cite{st} without proof;

\begin{corollary}
There exists an almost positive knot of type I which does not admit an almost positive diagram of type II.
\end{corollary}

A similar argument can be applied to establish the same equality for certain type of not-good successively almost positive diagrams.

\subsection{(Good) Successively almost positive v.s. strongly quasipositive}

A strongly quasipositive knots are not necessarily successively almost positive;  for a given link $L$, there exists a strongly quasipositive link $L'$ such that $\nabla_L(z)=\nabla_{L'}(z)$ \cite[88 Corollary]{ru}. See \cite[Section 5]{si} for details and concrete examples. Thus the Conway polynomial of strongly quasipositive link may not be non-negative.

\subsection{Questions}

We summarize the relations among various positivities discussed in the paper;

\[
\begin{array}{ccc}
\{\mbox{Positive}\} & & \\
\rotatebox[origin=c]{-90}{$\subsetneqq$}& & \\
\{\mbox{Almost positive of type I}\}& {\subset}_{(a)} & \{\mbox{Almost positive}\}\\
\cap_{(b)}& &\rotatebox[origin=c]{-90}{$\subset$}_{(c)}\\
\{\mbox{Good successively almost positive} \} & \subset_{(d)} &\{\mbox{Successively almost positive}\} \\
\rotatebox[origin=c]{-90}{$\subsetneqq$}& (\rotatebox[origin=c]{225}{$\subsetneqq$}?)& \\ 
\{\mbox{Strongly quasipositive} \} & &
\end{array}
\]

\begin{question}
Are the inclusions (a), (b), (c), (d) strict?
\end{question}

The strictness question of (a) appeared in \cite[Question 3]{st}.
We have seen that
\[ \{\mbox{Almost positive of type I}\} \not \subset \{\mbox{Almost positive of type II}\}\]
We are asking whether the converse inclusion 
\[\{\mbox{Almost positive of type II}\} \subset \{\mbox{Almost positive of type I}\}\]
holds or not.

As for the strictness of (b), one candidate of the properties of almost positive links which are not extended to good successively almost positive links is \cite[Theorem 5]{st}; it says that $\min \deg V_K(t)=\frac{1}{2}(1-\chi(K))$ when $K$ is almost positive. Although several simple successively almost positive links still satisfy the same property, the proof of the property utilizes a state-sum argument which we cannot generalize to successively almost positive case in an obvious manner.
 
\begin{question}
Is it true that $\min \deg V_K(t)=\frac{1}{2}(1-\chi(K))$ when $K$ is good successively almost positive?
\end{question}

As for the strictness of (d), it is interesting to ask to what extent a successively almost positive link shares the same properties as good successively almost positive links;

\begin{question}
Assume that $K$ is successively almost positive.
\begin{itemize}
\item Is $K$ strongly quasipositive?
\item Is $\max \deg_{z} \nabla_K(z)=1-\chi(K)$ when $K$ is non-split? 
\item Is $\max \deg_z P_K(v,z) = \min \deg_v P_K(v,z)= 1-\chi(K)$?
\end{itemize}
\end{question}

In Theorem
\ref{cor:finite-concordance} we showed the finiteness of successively almost positive knots in a topological concordance class, under the assumption that the topological concordance genus is sufficiently large. It is natural to investigate whether this additional assumption is necessary or not.

\begin{question}
Does every topological concordance class contain at most finitely many successively $k$-almost positive knots?
\end{question}

In \cite{oz} Ozawa showed the visibility of primeness for positive diagrams; a link represented by a positive diagram $D$ is non-prime if and only if the diagram $D$ is non-prime. 

\begin{question}
If a knot $K$ prime if $K$ is represented by a prime, good successively $k$-almost positive diagram?
\end{question}

\end{document}